\documentclass[11pt, oneside]{amsart}

\usepackage[margin=3.75cm]{geometry}
\usepackage{tabularx, hyperref}
\usepackage{amssymb} \usepackage{amsfonts} \usepackage{amsmath}
\usepackage{amsthm}
\usepackage{epsfig}
\usepackage{ amscd, amsxtra, latexsym}
\usepackage[all]{xy}
\usepackage{enumerate}
\usepackage{color}

\newtheorem{lemma}{Lemma}[section]
\newtheorem{thm}[lemma]{Theorem}
\newtheorem{prop}[lemma]{Proposition}
\newtheorem{cor}[lemma]{Corollary}

\theoremstyle{definition}
\newtheorem{defn}[lemma]{Definition}

\newtheorem{rem}[lemma]{Remark}

\theoremstyle{definition}

\definecolor{darkgreen}{cmyk}{1,0,1,.2}

\newcommand{\R} {\ensuremath {\mathbb{R}}}

\newcommand{\calC} {\ensuremath {\mathcal{C}}}

\newcommand{\calS} {\ensuremath {\mathcal{S}}}

\address{Department of Mathematics and Statistics, McGill University, Burnside Hall, 805 Sherbrooke Street West, Montreal, QC, Canada H3A 0B9}\address{Institute of Mathematics of the Polish Academy of Sciences, \'{S}niadeckich 8, 00-656 Warsaw, Poland }
\email{piotr.przytycki@mcgill.ca}
\address{Department of Mathematics, ETH Zurich, 8092 Zurich, Switzerland}
\email{sisto@math.ethz.ch}

\begin{document}

\title{A note on acylindrical hyperbolicity of Mapping Class Groups}
\author{Piotr Przytycki}
\author{Alessandro Sisto}

\begin{abstract}
The aim of this note is to give the simplest possible proof that Mapping Class Groups of closed hyperbolic surfaces are acylindrically hyperbolic, and more specifically that their curve graphs are hyperbolic and that pseudo-Anosovs act on them as loxodromic WPDs.
\end{abstract}

\maketitle

\section{Introduction}
Following Osin~\cite{Os-acyl}, we say that a group is
\emph{acylindrically hyperbolic} if it is not virtually cyclic
and it acts on a (Gromov-)hyperbolic space non-trivially in the
sense that there is a loxodromic WPD (weakly proper
discontinuous) element. Recall that an element $g$ of a group
acting on a metric space $X$ is \emph{loxodromic} if for each $x\in X$ there
exists $\epsilon>0$ such that for any integer
$n$ we have $d_X(x,g^nx)\geq \epsilon |n|$. See
Subsection~\ref{sec:WPD} for the definition of a WPD element, a
notion due to Bestvina--Fujiwara~\cite{BeFu-wpd}.

Acylindrical hyperbolicity has strong consequences: All
acylindrically hyperbolic groups are SQ-universal, contain free
normal subgroups~\cite{DGO}, contain Morse elements and hence
have cut-points in all asymptotic cones~\cite{Si-hypembmorse},
and have infinite dimensional bounded cohomology in
degrees~2~\cite{HullOsin} and~3~\cite{FPS-H3b-acylhyp}.
Moreover, acylindrically hyperbolic groups without finite
normal subgroups have simple reduced $C^*$-algebra~\cite{DGO}
and their commensurating endomorphisms are inner
automorphisms~\cite{MOS-pwise-inner}.

Mapping Class Groups of closed surfaces of genus at least~$2$
are among the motivating examples of acylindrically hyperbolic
groups. A natural hyperbolic space on which the Mapping Class
Group $\mathrm{MCG}(S)$ of the surface $S$ as above acts is the
\emph{curve graph} $\calC(S)$ of $S$. The vertices of
$\calC(S)$ are isotopy classes of essential simple closed
curves on $S$, and two isotopy classes are joined by an edge
(of length~1) if they contain disjoint representatives. From
now on a \emph{curve} means the isotopy class of an essential
simple closed curve on $S$, and we say that two curves are
\emph{disjoint} (resp.\ \emph{intersecting $\leq M$ times}) if
they can be represented disjointly (resp.\ intersecting $\leq
M$ times).

The fact that $\calC(S)$ is hyperbolic was proved by
Masur--Minsky~\cite{MM1}, who also proved that any
pseudo-Anosov element acts loxodromically, while the fact that
any pseudo-Anosov is WPD is due to
Bestvina--Fujiwara~\cite{BeFu-wpd}. (Bowditch proved a stronger
property, namely acylindricity~\cite{Bow-tight}.)

As it turns out, there are relatively simple proofs of all
these facts, which we present in this note. We think that a
result as important as the acylindrical hyperbolicity of
Mapping Class Groups deserves such a(n almost) self-contained
account.

In Section~\ref{sec:hyp} we show that curve graphs are
hyperbolic (Theorem~\ref{thm:hyp}), while in
Section~\ref{loxWPD} we show that pseudo-Anosovs are loxodromic
(Corollary~\ref{lox}) and~WPD (Corollary~\ref{WPD}).

\section{Curve graphs are hyperbolic}\label{sec:hyp}

Here is the main theorem of the section.

\begin{thm}\label{thm:hyp}
Let $S$ be a closed orientable surface of genus at least~2.
Then its curve graph $\calC(S)$ is hyperbolic.
\end{thm}

Our proof is inspired by the one in \cite{HPW-hyp}. It also
yields uniform hyperbolicity, but we do not record it in this
note.

We will actually prove hyperbolicity of the \emph{augmented
curve graph} $\calC_{{\mathrm{aug}}}(S)$, the metric graph with
the same vertex set as $\calC(S)$ and edges connecting pairs of
curves that intersect at most twice. Notice that
$\calC_{{\mathrm{aug}}}(S)$ is quasi-isometric to $\calC(S)$.
This follows from the fact that two curves that intersect at
most twice have a common neighbour in $\calC(S)$: the
neighbourhood of their union is a non-closed surface of Euler
characteristic $\geq -2$, whence of genus $\leq 1$, and thus
has an essential boundary component.

We will use the following criterion for hyperbolicity due to
Masur--Schleimer~\cite{MS-diskcplx}. A one page proof is
available in~\cite{Bow-unifhyp}. Here $N_D(\eta)$ denotes the
$D$-neighbourhood of a subset $\eta$ of a metric space.

\begin{prop}\label{guessgeod}
Let $X$ be a metric graph, and let $D\geq 0$. Suppose that to
every pair of vertices $x,y\in X^{(0)}$ we have assigned a
connected subgraph $\eta(x,y)$ containing $x$ and $y$ in such a
way that
 \begin{enumerate}
  \item for all $x,y\in X^{(0)}$ with $d_X(x,y)\leq 1$ we
      have $\mathrm{diam}_X(\eta(x,y))\leq D$,
  \item for all $x,y,z\in X^{(0)}$ we have
      $\eta(x,y)\subseteq N_D\big(\eta(x,z)\cup
      \eta(z,y)\big)$.
 \end{enumerate}
Then $X$ is hyperbolic.
\end{prop}

In our proof of Theorem~\ref{thm:hyp}, the sets $\eta(\cdot,
\cdot)$ will be spanned by the following curves.

\begin{defn}(Bicorn curves)
Let $a$ and $b$ be curves on $S$. Consider their
representatives on $S$ in minimal position (i.e.\ intersecting
a minimal number of times); we will call them also $a$ and $b$
slightly abusing our convention. Recall that minimal position
is unique up to isotopy.

A curve $c$ is a \emph{bicorn curve between $a$ and $b$} if
either $c=a$ or $c=b$ or $c$ is represented by the union of an
arc $a'$ of $a$ and an arc $b'$ of $b$, which we call the
\emph{$a$-arc} and the \emph{$b$-arc} of $c$. If $c=a$, then
its $a$-arc is $a$ and its $b$-arc is empty, similarly if
$c=b$, then its $b$-arc is $b$ and its $a$-arc is empty.
\end{defn}

Notice that whenever we have arcs $a'$ of $a$ and $b'$ of $b$
that only intersect at their endpoints, such arcs define a
(bicorn) curve which is essential, because $a$ and $b$ are in
minimal position.

Also, notice that, given $a$ and $b$, there are only finitely
many bicorn curves between $a$ and $b$.

\begin{lemma}\label{conn}
Let $a$ and $b$ be curves on $S$, and let $\eta(a,b)$ be the
full subgraph in $\calC_{\mathrm{aug}}(S)$ spanned by all
bicorn curves between $a$ and $b$. Then $\eta(a,b)$ is
connected.
 \end{lemma}

In the proof, we will use the following order on the set of
bicorn curves.

\begin{defn}
Fix curves $a,b$. For $c$ and $c'$ bicorn curves between $a$
and~$b$, we write $c<c'$ if the $b$-arc of $c'$ strictly
contains the $b$-arc of $c$.
\end{defn}

Note that since theoretically (but not \emph{de facto}) the
same bicorn curve $c$ might be represented in two different
ways as the union $a'\cup b'$ of an $a$-arc~$a'$ and a $b$-arc
$b'$, the order above is in fact an order on the
representatives of bicorn curves of form $a'\cup b'$.

\begin{proof}[Proof of Lemma \ref{conn}.] All bicorn curves in
this proof are bicorn curves between $a$ and $b$.

We claim that if $c$ is a bicorn curve and $c\neq b$, then
there exists a bicorn curve $c'$ such that $c< c'$ and $c'$ is
a neighbour of $c$ in $\calC_{{\mathrm{aug}}}(S)$. Since the
set of bicorn curves is finite, the claim implies that
$\eta(a,b)$ is connected.

\begin{figure}[h]
 \includegraphics[scale=0.9]{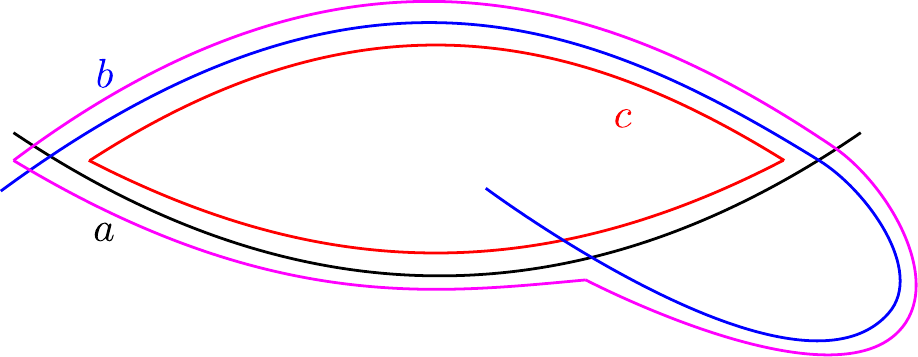}
\end{figure}

We now justify the claim. Let $c\neq b$ be a bicorn curve. First, if
$c$ and $b$ intersect at most twice, then we can take $c'=b$,
so let us assume that this is not the case.

If $c=a$, then let $b'$ be a minimal arc of $b$ whose both
endpoints lie in $a$. Let $c'$ be any of the two bicorn curves
defined by $b'$ and an arc of $a$ with the same endpoints. It
is easy to check that the intersection number of $a$ and $c'$
is at most 1, and clearly $c < c'$.

If $c\neq a$ with $a$-arc $a'$ and $b$-arc $b'$, then consider
a minimal arc $b''\supsetneq b'$ with both endpoints in $a'$.
The arc $b''$ and the subarc $a''$ of $a'$ with the same
endpoints as $b''$ define a new bicorn curve $c'$ with $c < c'$
and intersection number at most 1 with $c$. This justifies the
claim.
\end{proof}

The following lemma says that bicorn curve triangles are
1-slim.

\begin{lemma}\label{1thin}
Let $a,b$ and $d$ be curves and let $c$ be a bicorn curve
between $a$ and $b$. Then there is a bicorn curve $c'$ between
$a$ and $d$ or between $b$ and $d$ that intersects $c$ at most
twice.
\end{lemma}

\begin{proof}
If the intersection number of $c$ and $d$ is at most $2$, then
we can take $c'=d$. Otherwise, again slightly abusing the
notation, we consider representatives $a,b$ and $d$ on $S$
pairwise in minimal position.

We claim that there is an arc $d'$ of $d$ intersecting $a'$ (or
$b'$) only at its endpoints and either intersecting $b'$
(resp.\ $a'$) at most once, or intersecting it exactly twice:
at the endpoints. Indeed, to justify the claim it suffices to
take the minimal arc $d'$ of $d$ with both endpoints in $a'$ or
both endpoints in $b'$; such an arc exists since $d$ intersects
$c$ at least $3$ times.

\begin{figure}[h]
 \includegraphics[scale=0.9]{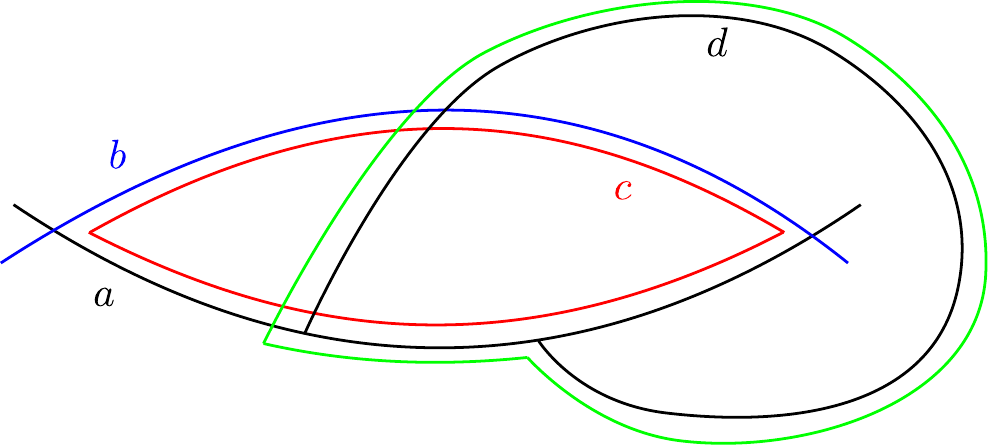}
\end{figure}

Let $d'$ be the arc guaranteed by the claim and assume without
loss of generality that its endpoints lie in $a'$. Then $d'$
and the subarc of $a'$ with the same endpoints as $d'$ define a
bicorn curve $c'$ between $a$ and $d$, and $c'$ intersects $c$
at most twice, as desired.
\end{proof}

\begin{proof}[Proof of Theorem \ref{thm:hyp}]
For vertices $a,b$ of $\calC_{\mathrm{aug}}(S)$, we define
$\eta(a,b)$ to be the full subgraph spanned in
$\calC_{\mathrm{aug}}(S)$ by the bicorn curves between~$a$ and~$b$, as in Lemma~\ref{conn}. Clearly, $\eta(a,b)$ contains $a$
and $b$. By Lemma~\ref{conn}, the subgraph $\eta(a,b)$ is
connected, as required in Proposition~\ref{guessgeod}.

Hypothesis~(1) of Proposition~\ref{guessgeod} is obviously
satisfied with $D=1$. Hypothesis~(2) of
Proposition~\ref{guessgeod} is satisfied with $D=1$ by
Lemma~\ref{1thin}. Thus by Proposition~\ref{guessgeod}, the
graph $\calC_{\mathrm{aug}}(S)$  is hyperbolic.
\end{proof}

\section{Pseudo-Anosovs are loxodromic WPD}\label{loxWPD}
\subsection{Notation}

Throughout the section we fix a pseudo-Anosov homeomorphism
$\phi:S\to S$, where $S$ is a closed surface. We will use some
well-known facts about pseudo-Anosovs discovered by
Thurston~\cite{Th-pA}. Recall that $\phi$ comes with a
one-parameter family of CAT(0) Euclidean metrics with
singularities $\{d_t\}_{t\in \R}$ on $S$. We denote the length
of the path $\alpha$ with respect to the metric $d_t$ by
$l_t(\alpha)$. The metrics $\{d_t\}$ have the following
properties:
\begin{itemize}
 \item the finitely many singularities of the metrics
     coincide and they are invariant under $\phi$.
 \item the metrics all have the same geodesics, up to reparametrisation.
 \item the push-forward of $d_t$ by $\phi$ is $d_{t+1}$.
 \item $S$ has two transverse singular foliations, called the horizontal and the vertical foliation, that form an angle of $\pi/2$ at every non-singular point with respect to any Euclidean structure $d_t$.
 \item $\phi$ preserves the horizontal and the vertical foliation, and there exists $\lambda>1$ so that if $\alpha$ is a subpath of the horizontal (resp.\ vertical) foliation then $l_t(\alpha)=\lambda^{t-t'}l_{t'}(\alpha)$ (resp. $l_t(\alpha)=(1/\lambda)^{t-t'}l_{t'}(\alpha)$).
 \item any half-leaf of either foliation is dense in $S$.
\end{itemize}

A \emph{saddle connection} is a geodesic on $S$ whose
intersection with the singular set consists of its endpoints.
Note that a saddle connection cannot be contained in the
horizontal or the vertical foliation. The \emph{balanced time} $\beta(\gamma)$ of a saddle connection
$\gamma$ is the only $t\in \R$ such that $l_t(\gamma)$ is
minimal, or, equivalently, the only $t$ such that $\gamma$
forms an angle of $\pi/4$ with the horizontal and vertical
foliations with respect to the metric $d_t$ at any point in its
interior.

\begin{defn}\label{baltime}
Given a curve $c$, we denote by $\calS(c)$ the set of all
saddle connections contained in the geodesic representative of
$c$ with respect to a (hence any) metric $d_t$, each counted
with multiplicity.

 Moreover, we denote by $\beta(c)$ the average over all $\gamma\in\calS(c)$ (counted with multiplicity) of the balanced time of $\gamma$.
\end{defn}

\begin{rem}\label{equivariance}
  The map $\beta$ is clearly $\phi$-equivariant, meaning that for every curve $c$ we have $\beta(\phi(c))=\beta(c)+1$.
\end{rem}

In fact, due to Lemmas~\ref{nocrosssamebal}
and~\ref{saddledontcross}, any function $\beta(c)$ with values
between $\min_{\gamma\in \calS(c)}\beta(\gamma)$ and
$\max_{\gamma\in \calS(c)}\beta(\gamma)$ satisfying
Remark~\ref{equivariance} would work for our purposes.

We will be interested in points of transverse intersection of
pairs of saddle connections. Note that an intersection point of
the saddle connections $\gamma_1$ and $\gamma_2$ is not a point
of transverse intersection if and only if it is either a common
endpoint of $\gamma_1$ and $\gamma_2$ or an interior point of
$\gamma_1=\gamma_2$.

\subsection{Preliminary lemmas}

In both lemmas below, part 2) will be a (technical)
generalisation of part 1). We will use parts 1) to show that
$\phi$ is loxodromic and part 2) to show that it is WPD. Hence,
on first reading, the reader may wish to read parts 1), then
the proof that $\phi$ is loxodromic and only afterwards move on
to parts 2) and the proof that $\phi$ is WPD.

\begin{lemma}\label{nocrosssamebal}
There exists a constant $C$ so that
\begin{enumerate}
 \item if the saddle connections $\gamma_1$ and $\gamma_2$ do not have points of transverse intersection, then $|\beta(\gamma_1)-\beta(\gamma_2)|\leq C$.
 \item for each $M$ there exists $D$ with the following
     property. If the saddle connections $\gamma_1$ and
     $\gamma_2$ satisfy
     $|\beta(\gamma_1)-\beta(\gamma_2)|\geq C$ and
     $l_{\beta(\gamma_1)}(\gamma_2)\geq D$, then $\gamma_1$
     and $\gamma_2$ have at least $M$ points of transverse
     intersection.
\end{enumerate}
\end{lemma}

\begin{proof}
We will say that a geodesic in a (hence any) metric $d_t$ is
\emph{Euclidean} if it intersects the singular set at most at
the endpoints. Notice that such a geodesic forms a well-defined
angle with the horizontal and vertical foliation with respect
to any metric $d_t$.

Let $\gamma_1$ and $\gamma_2$ be saddle connections. From now
on, lengths and angles will be measured in the metric $d_0$.

Up to using the action of $\phi$, we can assume that the
balanced time of $\gamma_1$ is within $0$ and $1$, so that (in
the metric $d_0$) the angle that $\gamma_1$ forms with both the
horizontal and the vertical foliation is bounded below by some
$\epsilon>0$ depending only on the pseudo-Anosov $\phi$. Also,
by the discreteness of the singular set, up to decreasing
$\epsilon$ we can assume that the length of $\gamma_1$ is $\geq
\epsilon$.

Now, by the compactness of $S$ and since half-leaves of both
foliations are dense, there exists $\delta>0$ so that any
Euclidean geodesic $\gamma$ with length $\geq 1/\delta$ forming
an angle $\leq \delta$ with either the horizontal or vertical
foliation intersects transversally
every
Euclidean geodesic of length $\geq \epsilon$ that forms an
angle $\geq \epsilon$ with both foliations.

We can now prove 1). Whenever $|\beta(\gamma_2)|$ is large
enough, the angle that $\gamma_2$ forms with one of the
foliations is $\leq \delta$, and by the discreteness of the
singular set the length of $\gamma_2$ is $\geq 1/\delta$. We can
thus apply the above statement with $\gamma=\gamma_2$ and
conclude that $\gamma_2$ intersects $\gamma_1$ transversally,
as desired.

In order to prove 2), notice that whenever $|\beta(\gamma_2)|$
is large enough and $\gamma_2$ is sufficiently long, then we
can split it into several $\gamma$'s as above.
\end{proof}

\begin{lemma}\label{saddledontcross}
\begin{enumerate}
 \item  If the curves $c_1$ and $c_2$ are disjoint
     (including the case $c_1=c_2$) then no saddle
     connection of $\calS(c_1)$ intersects transversally a
     saddle connection of $\calS(c_2)$.
 \item Let $c_1$ and $c_2$ be curves and let
     $\gamma^1_1,\dots,\gamma^k_1\in\calS(c_1),\gamma_2\in\calS(c_2)$
     be saddle connections so that $\gamma^i_1$ and
     $\gamma_2$ intersect transversally at $M_i$ points.
     Then the intersection number of $c_1$ and $c_2$ is at
     least $\sum M_i$.
\end{enumerate}
\end{lemma}

\begin{proof}
Let us prove 1) first.  Let $\alpha_1$ and $\alpha_2$ be the
geodesic representatives of $c_1$ and $c_2$. Suppose that
$\alpha_1$ contains a saddle connection that intersects
transversally a saddle connection contained in $\alpha_2$. We
can lift $\alpha_1$ and $\alpha_2$ to the universal cover of
$(S,d_0)$ to two bi-infinite geodesics $\tilde\alpha_1$ and
$\tilde\alpha_2$ that intersect transversally at a point. The
universal cover of $(S,d_0)$ is a CAT(0) space quasi-isometric
to $\mathbb{H}^2$. Thus $\tilde\alpha_1$ and $\tilde\alpha_2$
intersect exactly once and moreover the points at infinity of
$\tilde\alpha_1$ separate the points at infinity of
$\tilde\alpha_2$. Thus any representatives of
$c_1$ and $c_2$ intersect, a contradiction.

To prove 2), we can proceed similarly and this time find
$M=\sum M_i$ lifts $\tilde\alpha^1_1,\dots,
\tilde\alpha^M_1$ each intersecting transversally
$\tilde\alpha_2$, in distinct orbits of the stabiliser of
$\tilde\alpha_2$ in $\pi_1(S)$ (regarded as the group of
deck transformations). The conclusion
follows.
\end{proof}

\subsection{Pseudo-Anosovs are loxodromic}

\begin{prop}\label{balcorselip}
 The balanced time map $\beta:\calC(S)^{(0)}\to \R$ from Definition \ref{baltime} is coarsely Lipschitz, namely there exists $L\geq 1$ so that for all curves $c_1$ and $c_2$ we have
 $$|\beta(c_1)-\beta(c_2)|\leq Ld_{\calC(S)}(c_1,c_2).$$
\end{prop}

\begin{proof}
It is enough to show that there exists $L$ so that
$|\beta(c_1)-\beta(c_2)|\leq L$ when $c_1$ and $c_2$ are
disjoint. By Lemma \ref{saddledontcross}, no $\gamma\in
\calS(c_i)$ intersects transversally a $\gamma'\in \calS(c_j)$,
for $i,j\in\{1,2\}$. Hence, we get
$|\beta(\gamma)-\beta(\gamma')|\leq C$ for $C$ as in Lemma
\ref{nocrosssamebal}. Thus we can take $L=C$.
 \end{proof}

\begin{rem}\label{samecurvesamebal}
Notice that we have also just proved that for any curve $c$ and
any $\gamma\in \calS(c)$ we have $|\beta(\gamma)-\beta(c)|\leq
C$ for $C$ as in Lemma \ref{nocrosssamebal}.
 \end{rem}

\begin{cor}\label{lox}
 $\phi$ acts loxodromically on the curve graph.
\end{cor}

\begin{proof}
Recall from Remark \ref{equivariance} that $\beta$ is
$\phi$-equivariant, meaning that for every curve $c$ we have
$\beta(\phi(c))=\beta(c)+1$. Hence, in view of Proposition
\ref{balcorselip}, for any fixed curve $c$ and integer $n$ we
have
$$d_{\calC(S)}(c,\phi^n(c))\geq |\beta(\phi^n(c))-\beta(c)|/L=|n|/L,$$
so that $\phi$ is loxodromic as required.
\end{proof}

\subsection{Pseudo-Anosovs are WPD}\label{sec:WPD}

To simplify notation, we will denote $c_n=\phi^n(c)$ for any
curve $c$.

\begin{prop}\label{pAWPD}
For every curve $c$ and $R\geq 0$ the following holds. For
every sufficiently large $n$ there are only finitely many
curves of the form $hc$ or $hc_n$ where $h\in MCG(S)$ satisfies
$d_{\calC(S)}(c,hc)\leq R$ and $d_{\calC(S)}(c_n,hc_n)\leq R$.
\end{prop}

\begin{proof}
Let us show finiteness of the set of curves of the form $hc_n$,
since finiteness of the set of curves of the form $hc$ can be
proven symmetrically.

Our aim is to show that whenever $n$ is large enough we can
bound $l_t(hc_n)$ (meaning the length of the geodesic
representative) with $t=\beta(c)$. Once we do this, we get that
there are finitely many possible $hc_n$'s since there are only
finitely many geodesics whose length with respect to $d_t$ is
bounded above by any given constant.

We now have to choose constants carefully. Let $C$ be as in
Lemma~\ref{nocrosssamebal} and~$L$ be the coarsely Lipschitz
constant for $\beta$ as in Proposition~\ref{balcorselip}.
Suppose that $n$ is large enough so that $d_{\calC(S)}(c,c_n)\geq
2L^2R+3LC$, and let $M-1$ be the intersection number of $c$ and
$c_{n}$. Finally, let $D$ be as in Lemma~\ref{nocrosssamebal}
for the given $M$, and let $D'=\lambda^{LR+C}MD$.

Now, assume by contradiction $l_t(hc_n)\geq D'$. Notice that by
Proposition~\ref{balcorselip} we have $|\beta(hc)-\beta(c)|\leq
LR$ and $|\beta(hc_n)-\beta(c_n)|\leq LR$, so that
$$|\beta(hc)-\beta(hc_{n})|\geq |\beta(c)-\beta(c_{n})|-2LR\geq d_{\calC(S)}(c,c_{n})/L-2LR\geq 3C.$$
In particular, for every $\gamma\in \calS(hc)$ and $\gamma_n\in
\calS(hc_n)$ we have $|\beta(\gamma)-\beta(\gamma_n)|\geq C$ by
Remark~\ref{samecurvesamebal}.

First, suppose that $\calS(hc_n)$ contains at least $M$
elements. Notice that each saddle connection in $\calS(hc_n)$
intersects transversally any saddle connection in $\calS(hc)$
by Lemma~\ref{nocrosssamebal}-(1). Hence, by
Lemma~\ref{saddledontcross}-(2) the intersection number of $hc$
and $hc_n$ is at least $M$, a contradiction.

Otherwise, there exists $\gamma_n\in\calS(hc_n)$ such that
$l_{t}(\gamma_n)\geq D'/M$. Pick any $\gamma\in \calS(hc)$.
Then, since $|t-\beta(\gamma)|\leq
|\beta(c)-\beta(hc)|+|\beta(hc)-\beta(\gamma)|\leq LR+C$ (by
Proposition~\ref{balcorselip} and
Remark~\ref{samecurvesamebal}), we have
$$l_{\beta(\gamma)}(\gamma_n)\geq l_t(\gamma_n) \lambda^{-LR-C}\geq D.$$
Hence, by Lemma~\ref{nocrosssamebal}-(2), $\gamma$ and
$\gamma_n$ intersect transversally at least $M$ times, so that
by Lemma \ref{saddledontcross}-(2) the intersection number of
$hc$ and $hc_n$ is at least $M$, a contradiction.
\end{proof}

Recall that $\phi\in \mathrm{MCG}(S)$ is \emph{WPD} (for the
action on $\calC(S)$) if, given any curve $c$ and any $R\geq
0$, for any sufficiently large $n$ there are only finitely many
elements $h\in \mathrm{MCG}(S)$ satisfying
$d_{\calC(S)}(c,hc)\leq R$ and $d_{\calC(S)}(c_n,hc_n)\leq R$.

\begin{cor}\label{WPD}
$\phi$ is WPD.
\end{cor}

\begin{proof}
Fix any curve $c$ and any $R\geq 0$. For $n$ sufficiently large
the curves $c$ and $c_n$ form a filling pair, that is, the
complementary regions of the union of the representatives of
$c$ and $c_n$ in minimal position are discs. By
Proposition~\ref{pAWPD}, there are finitely many possibilities
for $b=hc$ and $b_n=hc_{n}$. Since the complementary regions of
$c\cup c_n$ are disks, for fixed $b$ and $b_n$ there are only
finitely many $h$ satisfying $hc=b, hc_n=b_n$. Hence, there are
finitely many possibilities for $h$, as required.
 \end{proof}

\bibliographystyle{alpha}
\bibliography{biblio.bib}

\end{document}